\documentclass[12pt]{amsart}
\usepackage[margin=1in]{geometry}
\usepackage{graphicx} 
\usepackage{tikz}
\usetikzlibrary{shapes}
\usetikzlibrary{intersections}
\usetikzlibrary{svg.path}
\usetikzlibrary{decorations.markings}
\usetikzlibrary{hobby}
\usepackage{amsmath,mathrsfs}
\usepackage{physics}
\usepackage{amssymb,amsrefs}
\usepackage{bm}
\usepackage{dsfont}
\usepackage{hyperref}

\theoremstyle{plain} 
\newtheorem{theorem}{Theorem}[section]
\newtheorem*{theorem*}{Theorem}
\newtheorem{lemma}[theorem]{Lemma}
\newtheorem*{lemma*}{Lemma}

\newtheorem*{corollary*}{Corollary}
\newtheorem{proposition}[theorem]{Proposition}
\newtheorem*{proposition*}{Proposition}

\newtheorem{definition}[theorem]{Definition}
\newtheorem*{definition*}{Definition}

\newtheorem*{example*}{Exemple}

\newtheorem*{remark*}{Remark}
\newtheorem*{remarks*}{Remarks}

\newcommand{\C}{\mathbb{C}}

\newcommand{\R}{\mathbb{R}}

\newcommand{\beq}{\begin{equation}}
\newcommand{\eeq}{\end{equation}}

\renewcommand{\th}{\theta}
\newcommand{\g}{\gamma}
\newcommand{\G}{\Gamma}
\newcommand{\e}{\epsilon}
\newcommand{\N}{\mathcal{N}}
\newcommand{\sign}{\textrm{sgn}}

\newcommand{\E}[1]{\mathbb{E}\left[{#1}\right]}
\renewcommand{\tr}{\mathrm{tr}}
\renewcommand\dd{\mathop{}\!\mathrm{d}}
\newcommand{\T}[3]{T_{{#1},{#2}}\left(#3\right)}
\renewcommand{\l}{\lambda}
\renewcommand{\a}{\alpha}
\renewcommand{\b}{\beta}
\newcommand{\s}{\sigma}
\renewcommand{\t}{\tau}
\renewcommand{\bar}{\overline}
\newcommand{\Fp}[2]{F_{#1}^+(#2)}
\newcommand{\Fm}[2]{F_{#1}^-(#2)}
\DeclareRobustCommand{\sfir}{\genfrac[]{0pt}{}}
\DeclareRobustCommand{\ssec}{\genfrac\{\}{0pt}{}}

\usepackage{empheq}

\title[On a random matrix proof of a bipartite Harer-Zagier formula]{On a random matrix proof of\\ a bipartite Harer-Zagier formula}
\author{Guillaume Dubach}
\author{Hai An Mai}
\address{Centre de Mathématiques Laurent Schwartz, École polytechnique, Institut Polytechnique de Paris, Palaiseau, France}
\email{guillaume.dubach@polytechnique.edu}
\thanks{The authors gratefully acknowledge support from the Fondation de l'École polytechnique, as well as from the Agence Nationale de la Recherche: ANR-25-CE40-5672 and ANR-25-CE40-1380.}

\date{June 2026}

\usepackage{enumerate}
\usepackage{thmtools, thm-restate}

\begin{document}

\begin{abstract}
    This work establishes a bipartite generalization of the Harer-Zagier formula using non-Hermitian Random Matrix Theory. More specifically, we use a decomposition of powers of Ginibre eigenvalues as a superposition of independent point processes to identify all coefficients of the generating function of the genus of a surface obtained by a random bipartite pairing of the sides of one polygon with $kM$ sides and $k$ polygons with $M$ sides.
\end{abstract}

\maketitle

\vspace{.05in}

\tableofcontents

\vspace{.05in}

\section{Introduction}

\vspace{.05in}

In the celebrated work \cite{HarerZagier}, Harer and Zagier endeavored to study the Euler characteristic of moduli spaces of curves; this boils down to evaluating the numbers $\e_{g}(n)$, which are defined as the number of ways a surface of genus $g \geq 0$ can be obtained by a pairing of the sides of one polygonal face with $2n$ sides (as represented on Figure \ref{fig:random-general}). Using techniques from hermitian Random Matrix Theory, especially a remarkable three-term recurrence on the moments of a GUE matrix, they obtain several equivalent characterizations of these numbers $\e_{g}(n)$. For instance, the generating series of these coefficients is given by the following formula:

\begin{theorem}[Harer-Zagier generating function \cite{HarerZagier}]\label{thm_HZ_gen_fun}
    Let $n$ be a non-negative integer; then the following equality holds:
    \[\sum_{g=0}^{\lfloor n/2\rfloor} \e_g(n) \cdot N^{n+1-2g} = (2n-1)!! \sum_{l=0}^{n} 2^{l}\binom{n}{l}\binom{N}{l+1}\]
    for all integers $N$.
\end{theorem}

\begin{figure}[ht]
    \centering
    \includegraphics[width=6cm]{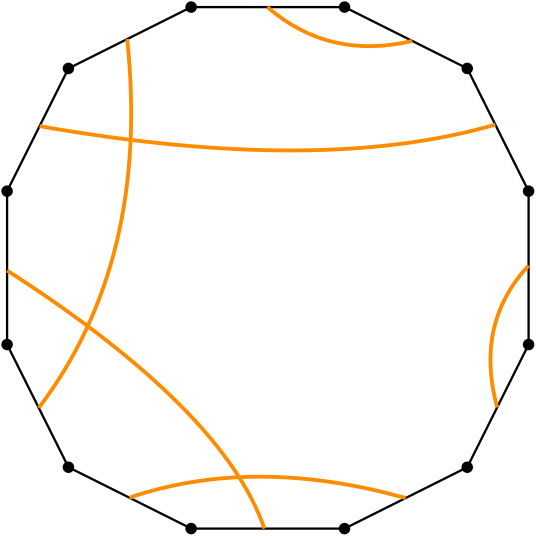}
    \caption{A pairing between the sides of a $12$-gon ($n=6$).}
    \label{fig:random-general}
\end{figure}

\newpage
\noindent Equivalently, the coefficients $\e_g(n)$ can be expressed with a closed form:

\begin{theorem}[Harer-Zagier coefficients]\label{thm_HZ_coefficient}
    For $n,g \geq 1$ we have the following equality
    \[\e_g(n) = (2n-1)!!\cdot 2^{n-2g}\sum_{p=0}^{2g} \frac{(-2)^p}{(n+1+p-2g)!} \binom{n}{n+p-2g}\sfir{n+1+p-2g}{n+1-2},\]
    where $\sfir{n}{k}$ is the unsigned Stirling number of the first kind.
\end{theorem}

The present work deals with a bipartite version of the above formulas. Let $k$ and $M$ be positive integers. Given a $kM$-gon and $k$ distinct $M$-gons, we consider all bipartite edge pairings between the $kM$-gon and the $k$ distinct $M$-gons; by bipartite, we mean that no pairing is allowed within the $kM$-gon, or between the smaller $M$-gons; this is represented on Figure \ref{fig:random-bipartite}. With the same interpretation of edge pairings as before, we identify them as one edge and connect the faces without any twists to form an orientable surface. The question is then to study the genus of a~surface obtained in this way. We will denote the number of such pairings leading to a surface of genus $g$ as $\e_g(k,M)$.

\begin{figure}[ht]
    \centering
    \includegraphics[width=7.5cm]{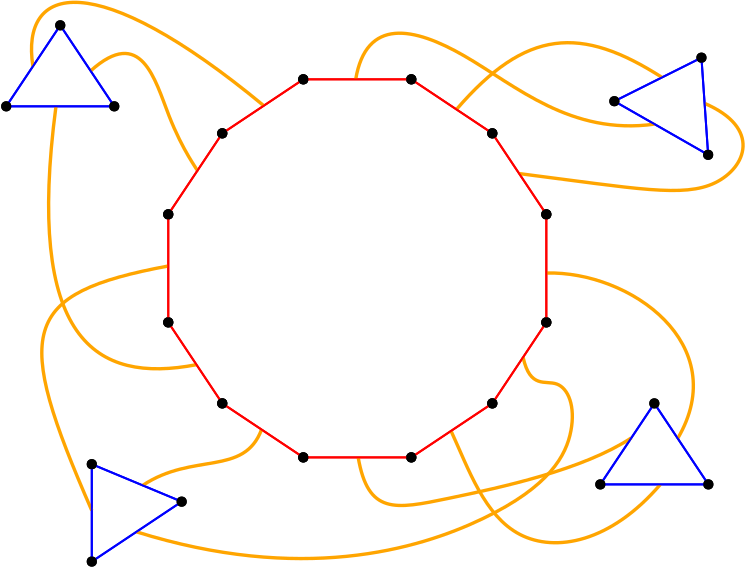}
    \caption{Example of a bipartite pairing between the edges of a dodecagon and those of four triangles, i.e. $k=4$, $M=3$.}
    \label{fig:random-bipartite}
\end{figure}

\noindent We establish the following formula for the generating function of these coefficients:

\begin{theorem}[Bipartite generating function]\label{thm_bipartite_HZ_gen_fun}
    Let $k$ and $M$ be positive integers. Then we have the following equality
    \begin{equation}\label{eqn:main}
    \sum_{g =0}^{\lfloor k(M-1)/2\rfloor} \e_g(k,M) \cdot N^{k(M-1)+1-2g} = \frac{1}{kM+1} \sum_{l=0}^k (-1)^{l} \binom{k}{l} F_{kM+1}^+(N-lM).    
    \end{equation}
    for all integers $N$, where $F_{n}^+$ denotes the rising factorial (Def. \ref{def:risfalfact}).
\end{theorem}

\noindent We also derive an explicit closed-form expression for the number of pairings in Section~\ref{sec:reccoef}:

\begin{theorem}[Bipartite coefficients]\label{thm_bipartite_HZ_exact_coeff}
    For all integers $k,M >0$, $g \geq 0$, we have the following identity
    \[\e_g(k,M) = \frac{(-1)^k k!}{kM+1}\sum_{p=0}^{2g} \sfir{kM+1}{p+kM+1-2g}\binom{p+kM+1-2g}{p+k}\ssec{p+k}{k}(-M)^{p+k},\]
    where $\sfir{n}{k}$ and $\ssec{n}{k}$ are, respectively, Stirling numbers of the first and second kind.
\end{theorem}

\noindent In Section \ref{sec_recovering_HZ} we recover the Harer-Zagier formula as a particular case of the above bipartite setting, noticing the relation:
\[\e_g(k,2) = k! \cdot 2^k\cdot \e_g(k).\]
between the original Harer-Zagier coefficients $\e_g(k)$ and the bipartite coefficients $\e_g(k,M)$ with $M=2$.
These results were derived from Random Matrix techniques as part of the recent bachelor’s thesis work \cite{Mai2026}. Counting bipartite pairings is equivalent to counting the number of cycles in a commutator between random permutations with prescribed cycle structure, or the number of Grothendieck's \textit{dessins d'enfants} with prescribed degrees, which has also been achieved by other means -- especially combinatorial and representation-theoretic methods. Most importantly, Jackson's works \cite{Jackson1, Jackson2} allow to enumerate products of permutations with given number of cycles, and Chen \cite{Chen} recently recovered and extended this result using representation theory, while generalizing the Harer-Zagier formula by characterizing the number of products of cycles with specified conjugacy classes. The purpose of this work is to rederive equivalent expressions directly from RMT tools, and especially to illustrate the use of the complex Ginibre block-decomposition (Theorem \ref{thm:blocks}) for such a computation. One may hope to go further in this direction: indeed, the systematic use of non-hermitian RMT to tackle combinatorial problems seems particularly promising, due to the great progress made in recent years in the study of spectral moments of complex, real and quaternionic Ginibre ensembles \cites{AkemannByunOh2026,ForresterByun2024}. \\

\noindent The structure of the paper is as follows:
\begin{itemize}
\item Section \ref{sec_preliminaries} presents the statements and provides references for the random matrix results that are required in the proofs of our main theorems.
\item In Section \ref{sec_first_cases}, we analyze the initial cases in detail in order to illustrate the method.
\item In Section \ref{sec_main_result}, we establish the bipartite version of the Harer–Zagier formulas for arbitrary parameters $k$ and $M$, with the exact computation of the generating function (Theorem \ref{thm_bipartite_HZ_gen_fun}), and the derivation of the explicit closed-form expression of the coefficients $\e_g(k,M)$ (Theorem \ref{thm_bipartite_HZ_exact_coeff}).
\item Finally, in Section \ref{sec_recovering_HZ}, we show that our result indeed constitutes a genuine generalization, in the sense that the classical Harer–Zagier formula can be recovered as a~special case.
\end{itemize}

\subsection{Notations}
    We recall a few important definitions and properties that we rely on throughout the paper.

    \begin{definition}
        Let $\a$ be a positive real number. The Gamma distribution with parameter $\a$ is defined by the density
        \[\frac{1}{\Gamma(\alpha)} t^{\alpha-1}e^{-t} \]
        with respect to Lebesgue measure on $\R_+$,
        where $\G$ is the Gamma function, defined as
        \[\Gamma(\alpha) = \int_{\R_+} t^{\alpha-1}e^{-t}\dd t.\]
    \end{definition}
    
    \begin{lemma}\label{lem:gamma}
        Let $\g_{\a}$ be a gamma variable with parameter $\a$, then the $M$-th moment of $\g_{\a}$ is given by
        \[\E{(\g_{\a})^M} = \frac{\G(M+\a)}{\G(\a)}. \]
    \end{lemma}
    \begin{proof}
        By direct computation,
        \[\E{\g_{\a}^M} = \frac{1}{\G(\a)}\int_{\R_+} t^M t^{\a-1}e^{-t} \dd t = \frac{1}{\G(\a)}\int_{\R_+}t^{(M+\a)-1} e^{-t}\dd t = \frac{\G(M+\a)}{\G(\a)}.\]
    \end{proof}

    We also introduce the following notations for rising and falling factorials.

    \begin{definition}\label{def:risfalfact}
        Let $n$ be a non-negative integer. We call $F_n^+$ and $F_n^-$ the rising and falling factorial, defined as
        \begin{align*}
        F_n^+(X) & = X(X+1)\cdots(X+n-1) \\
        F_n^-(X) & = X(X-1)\cdots(X-n+1) = (-1)^n F_n^+(-X).
        \end{align*}
        By convention we denote $F_0^+ = F_0^- = 1$. Note that by definition $F_n^+$ and $F_n^-$ are polynomials of degree $n$.
    \end{definition}

    The following lemma will be used multiple times.
    \begin{lemma}\label{lem:rising-sum}
        Let $n$ and $k$ be positive integers, then
        \[\Fp{n}{k} = n \cdot \sum_{j=1}^k \frac{\G(n+j-1)}{\G(j)}.\]
    \end{lemma}
    \begin{proof}
        The derivation comes from the discrete differential operator:
        \[\partial : P(X) \mapsto P(X) - P(X-1).\]
        Note that
        \[\partial F_n^+ = X\cdots(X+n-1) - (X-1)\cdots(X+n-2)\]
        \[= X\cdots(X+n-2) \cdot[(X+n-1) - (X - 1)] = n F_{n-1}^+.\]
        Thus
        \[\Fp{n}{k} = \sum_{j=1}^k \partial \Fp{n}{j} = \sum_{j=1}^k n \Fp{n-1}{j} = n\sum_{j=1}^k \frac{\G(n+j-1)}{\G(j)}.\]
    \end{proof}

    \begin{definition}
        Let $n$ be a positive integer. The signed Stirling numbers of the first kind $s(n,k)$ are given by the coefficients of the falling factorial:
        \[F_{n}^-(X) = X(X-1)\cdots(X-n+1) = \sum_{k=0}^n s(n,k)X^k.\]
        The unsigned Stirling numbers of the first kind $\sfir{n}{k}$ are given by the coefficients of the rising factorial:
        \[F_{n}^+(X) = X(X+1)\cdots(X+n-1) = \sum_{k=0}^n \sfir{n}{k}X^k.\]
    \end{definition}
    
    \begin{definition}
        Let $n$ and $k$ be non-negative integers. The Stirling number of the second kind $\ssec{n}{k}$ is the number of ways to partition a set of $n$ labeled objects into $k$ non-empty subsets.
    \end{definition}

    \noindent We notice a similar link between Stirling numbers of the second kind and the falling factorials.

    \begin{proposition}\label{prop:stirling_second_falling_fact}
    Let $r$ be a non-negative integer, we have the relation:
        \[\sum_{m=0}^r \ssec{r}{m} F_m^-(X) = X^r.\]
    \end{proposition}

    \noindent We present two useful properties of Stirling numbers of the second kind:

    \begin{proposition}\label{prop:stirling_second_closed}
        Let $n$ and $k$ be non-negative integers, then
        \[\ssec{n}{k} = \frac{1}{k!} \sum_{j=0}^k (-1)^{k-j} \binom{k}{j}j^n.\]
    \end{proposition}

    \begin{lemma}\label{lem:stirling_second_power} 
        Let $r$ be a positive integer, then 
        \[\sum_{m=0}^r (-1)^m\cdot m! \cdot \ssec{r}{m} = (-1)^r.\]
    \end{lemma}
    \begin{proof}
        We will use the Proposition \ref{prop:stirling_second_falling_fact} with $X = -1$:
        \[(-1)^r = \sum_{m=0}^r \ssec{r}{m} F_m^-(-1)= \sum_{m=0}^r  (-1)^m m!\ssec{r}{m}.\]
    \end{proof}

\section{Preliminaries: Ginibre eigenvalues and genus expansion}\label{sec_preliminaries}

The present work relies on the properties of the complex Ginibre ensemble of random matrices. For a~positive integer $N$, let $G$ be an (unscaled) $N \times N$ Complex Ginibre matrix, i.e. a random matrix with i.i.d. Gaussian entries
     $$ G_{ij} \sim \N_\C(0,1) $$
for all $1 \leq i,j \leq N$. The fact that $G$ is a non-Hermitian matrix with complex coefficients is relevant to the purpose of modeling bipartite pairings with orientation-preserving identifications of edges. We build upon foundations that were laid especially in the works \cites{DubachPeled,DubachCycles,Ginibre1965,HKPV,Kostlan} among other sources, and recall here the theorems that will be needed in the next sections.

\subsection{Radii and powers of complex Ginibre eigenvalues}\label{sec:randmatrix}
Eigenvalues of Ginibre Ensembles have been studied extensively since the seminal 1965 paper \cites{Ginibre1965}, and we make use of some of their well-known properties. The first property we need is the independence of eigenvalue radii, a~result known as Kostlan's theorem:

    \begin{theorem}[Kostlan, \cite{Kostlan}]\label{thm:kostlan} Let $\{\l_1,\ldots,\l_N\}$ be the set of eigenvalues of Ginibre $G$, the following equality in distribution holds:
    \[\{|\l_1|^2,\ldots,|\l_N|^2\} \stackrel{d}{=} \{\g_1,\ldots,\g_N\},\]
    where $(\g_k)$ are independent gamma variables with parameters $1,\ldots,N$ respectively.
    \end{theorem}

    Another phenomenon is that high powers (i.e. with exponent $M \geq N$) are distributed as a set of independent variables with explicit distributions:

    \begin{theorem}[Hough, Krishnapur, Peres, Virág \cite{HKPV}]\label{thm:hkpv}
    Let $\{\l_1,\ldots,\l_N\}$ be the set of eigenvalues of a complex Ginibre matrix $G$. For any $M\geq N$, the following equality in distribution holds:
    \[\{\l_1^M, \ldots, \l_N^M\} \stackrel{d}{=}\{\g_1^{M/2}e^{i\th_1}, \ldots, \g_N^{M/2}e^{i\th_N}\},\]
    where $(\g_k), (\th_k)$ are independent and $(\th_k)$ are i.i.d. real variables uniform in $[0,2\pi]$ and $\g_1,\ldots,\g_N$ are gamma variables with corresponding parameters $1,\ldots,N$.
    \end{theorem}

More generally, it was proved in \cite{DubachPowers} that for any power $M$, the powers of the eigenvalues decorrelate to some extent: in the case of intermediate power, when $M < N$, one still finds a description of the spectrum as the superposition of $M$ independent point processes:

    \begin{theorem}[Block Decomposition \cite{DubachPowers}]\label{thm:blocks} Let $\{\l_1,\ldots,\l_N\}$ be the set of eigenvalues of a~complex Ginibre matrix $G$. For $M \geq 1$, the following equality in distribution holds:
    \[\{\l_1^M,\ldots,\l_N^M\} \stackrel{d}{=} \{z_1,\ldots, z_N\},\]
    where by denoting $I_p = \{i \in [N]\mid i \equiv p \pmod M\}$ for $1 \leq p \leq M$, $(z_i)_{i\in I_p}$ form independent blocks for different values of $p$, with joint density
    \[\frac{1}{Z_{N,M,p}} \prod_{\substack{i < j \\ i,j \in I_p}} |z_i-z_j|^2 \prod_{i\in I_p} |z_i|^{\frac{2(p-M)}{M}}e^{-|z_i|^{2/M}}\dd m(z_i),\]
    where
    \[Z_{N,M,p} = \pi^{c_p}c_p! M^{c_p} \prod_{j \in I_p} (j-1)!, ~ \text{and} ~~ c_p = |I_p|.\]
    \end{theorem}

\subsection{Genus expansion}
    The computations presented here relies on the technique called genus expansion; this is a well-known technique in Random Matrix Theory, going back to the pioneering work of Wigner \cites{Wigner2, Wigner}, used in the context of several-matrix models \cites{Capitaine}, in free probability \cites{Nica, Redelmeier}, as well as in enumeration of maps \cites{Zvonkin, Bouttier}. An exposition of this method is given in \cite{DubachPeled}, for words in complex Ginibre matrices specifically. Let $k$ and $M$ be positive integers. For the current purpose, we define

\begin{equation}\label{def_TkM}
T_{k,M}(N) := \E{\tr(G^{kM})\overline{\tr(G^M)^k}},
\end{equation}

where $N$ is the size of the Complex Ginibre Matrix $G$. Using genus expansion, one can see that $T_{k,M}(N)$ is a polynomial expression in $N$, and that it matches exactly the generating function of the coefficients $\epsilon_g(k,M)$:

    \begin{proposition}\label{prop:TkM}
        For any positive integers $k$ and $M$,
        \begin{enumerate}[(i)]
            \item $T_{k,M}$ is a polynomial in $N$ and it can be written as
            \[\displaystyle\T{k}{M}{N} = \sum_{g=0}^{\lfloor k(M-1)/2\rfloor} \e_g(k,M) \cdot N^{k(M-1)+1-2g},\]
            \item $\T{k}{M}{0} = 0$,
            \item $T_{k,M}$ is a polynomial in $N$ and $\deg(T_{k,M}) = k(M-1) + 1$,
            \item $T_{k,M}$ is odd if $2 \mid k(M-1)$, and even otherwise.
        \end{enumerate}
    \end{proposition}

    \begin{proof}
        This is a particular case of genus expansion for words in complex Ginibre matrices, as explained in the paper \cite{DubachPeled}. In this case the words are simply powers of one Ginibre matrix of size $N$. Thus, (i) is given directly by the genus expansion formula, the proof of which can be described in the following way: one writes down the expansion of the terms $\tr G^{kM}$ and $(\tr G^M)^k$ and then apply Wick's formula, transforming the expectation in \eqref{def_TkM} into a sum over pairings that has exactly the same structure as a bipartite pairing counted by the numbers $\epsilon_g(k,M)$.
        Properties (ii) and (iii) are then immediate consequences of this construction, as the exponent $k(M-1)+1-2g$ is the number of vertices on the resulting graph (after identifications), and this number is clearly an integer between $1$ and $k(M-1)+1$. 
        We also notice that $\e_0(k,M) \neq 0$ to prove the third item, as by inspection there is always at least one way to form a sphere from one $kM$-gon and $k$ separate $M$-gon (in fact these numbers are the Fuss-Catalan numbers \eqref{Fuss_Catalan}). For property (iv), we notice that all powers of $N$ are of the same parity as $k(M-1)+1$, thus the polynomial $T_{k,M}$ is either even or odd, depending on the values of $k$ and $M$.
    \end{proof}

\section{Bipartite formula: the first cases}\label{sec_first_cases}

We first provide a detailed computation of the first cases ($k=1,2,3,4$) of our main result. The common strategy to all cases is to evaluate the values $T_{k,M}(N)$ at more than $\deg(T_{k,M}) = k(M-1)+1$ values, in order to fully characterize the polynomial $T_{k,M}$; this requires a detailed analysis as $k$ grows larger, but for small values of $k$ the computation is simplified. The case $k=1$ was solved in \cite{DubachCycles}, which was the starting point of this project and of the bachelor thesis \cite{Mai2026}. It turns out that the second case ($k=2$) can be solved with very little extra work. The next cases $k=3$ and $k=4$ can also be solved by a similar computation, with an extra ingredient, namely the block decomposition \ref{thm:blocks}, thus opening the way to the general proof presented in section \ref{sec_main_result}.

\subsection{Case $k=1$} We consider bipartite pairings of two polygons of size $M$, as represented on Figure \ref{fig:case_M_1}.

    \begin{figure}[ht]
        \centering
        \includegraphics[width=6cm]{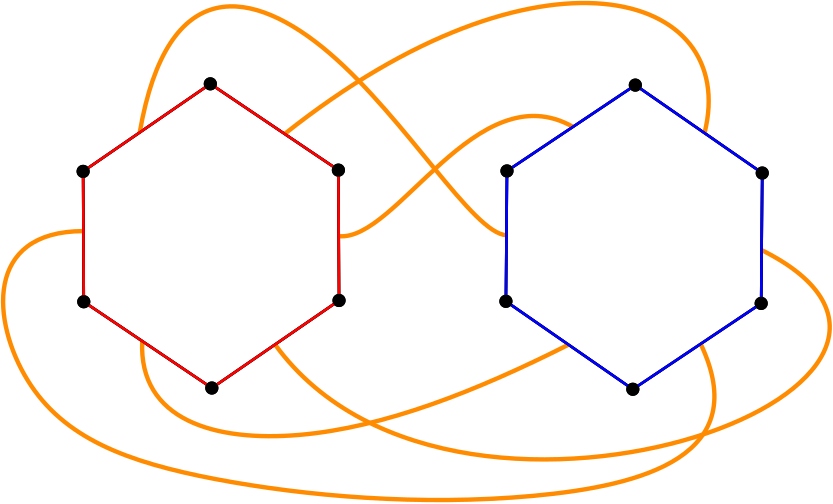}
        \caption{Example of a bipartite pairing of the sides of two hexagons i.e. $k=1$, $M=6$.}
        \label{fig:case_M_1}
    \end{figure}

It was noticed in \cite{DubachCycles} (although with a different goal and specific notations) that one can evaluate the polynomial $T_{1,M}$ directly from the distribution of high powers (Theorem \ref{thm:hkpv}) and square radii (Theorem \ref{thm:kostlan}). We give the details of this computation for the convenience of the reader.

    \begin{proposition}\label{prop:1M}
        For $0 \leq N \leq M$, we have the following equality
        \beq
        T_{1,M}(N) = \frac{1}{M+1}\Fp{M+1}{N}.
        \eeq
    \end{proposition}

    \begin{proof}
        Going back to the definition \eqref{def_TkM}, we find
        \[T_{1,M}(N) = \E{\tr(G^{M}) \bar{\tr(G^M)}} = \E{\sum_{i,j = 1}^N \l_i^{M} \bar{\l_{j}^M}} = \E{\sum_{i,j= 1}^N z_i \overline{z_{j}}},\]
        where $(z_j)_{j=1}^{N}$ are the variables from Theorem \ref{thm:hkpv} corresponding to the high powers $\l_j^M$ with $M\geq N$. Using the independence of these variables, we find that the contribution is non-zero only if $j=i$; therefore by Kostlan's theorem \ref{thm:kostlan}, Lemma \ref{lem:gamma} and finally Lemma \ref{lem:rising-sum}, we find that
        \[T_{1,M}(N) = \sum_{i=1}^N\E{|z_i|^{2}} = \sum_{i=1}^N \frac{\G(M+i)}{\G(i)} = \frac{1}{M+1}\Fp{M+1}{N},\]
        which was the claim.
    \end{proof}

    Proposition \ref{prop:1M} allows us to identify the polynomial $T_{1,M}$.

    \begin{proposition}
        Let $M$ be a positive integer. Then we have that for every integer $N$,
        \beq
        \T{1}{M}{N} = \frac{1}{M+1}[F_{M+1}^+(N) -F_{M+1}^-(N)] = \frac{1}{M+1}\left[\Fp{M+1}{N} - \Fp{M+1}{N-M}\right].
        \eeq
    \end{proposition}

    \begin{proof}
        From propositions \ref{prop:1M} and \ref{prop:TkM} we have that the polynomial
        \[\frac{1}{M+1}\cdot \Fp{M+1}{N} - \T{1}{M}{N}\]
        has $M+1$ zeros $\{0,1,\ldots,M\}$, and it is of degree $M+1$, hence
         \[\frac{1}{M+1}\cdot \Fp{M+1}{N} - \T{1}{M}{N} = C \cdot N(N-1)\cdots(N-M) = C\cdot \Fm{M+1}{N}.\]
        By inspection of the leading coefficient, we find the constant $C=\frac{1}{M+1}$ and conclude that
        $$
\T{1}{M}{N} = \frac{1}{M+1}[F_{M+1}^+(N) -F_{M+1}^-(N)].
        $$
        The second part of the claim comes from the correspondence between rising and falling factorial:
        $$
        F_{M+1}^-(X) = F_{M+1}^+(X-M).
        $$
    \end{proof}

This settles the case $k=1$. At that point, let us notice that we can generalize Proposition \ref{prop:1M} to every $k$ and find that the values of $T_{k,M}$ in the integer interval $[\![1,M]\!]$ are exactly given by those of a rising factorial. However, this fact alone is not sufficient to characterize the polynomial $T_{k,M}$ in general.

\begin{proposition}\label{prop:kM_1M}
        For any $k \geq 1$ and $0 \leq N \leq M$, we have the following equality
        \[T_{k,M}(N) = \frac{1}{kM+1}\Fp{kM+1}{N}.\]
    \end{proposition}

    \begin{proof}
        Similarly to the proof of Proposition \ref{prop:1M}, we notice
        \[T_{k,M}(N) = \E{\tr(G^{kM}) \bar{\tr(G^M)^k}} = \E{\sum_{i,j_1,\ldots,j_k = 1}^N \l_i^{kM} \bar{\l_{j_1}^M}\cdots\bar{\l_{j_k}^M}} = \E{\sum_{i,j_1,\ldots,j_k = 1}^N z_i^k \overline{z_{j_1}}\ldots \overline{z_{j_k}}},\]
        where $(z_n)_n$ are the independent variables from Theorem \ref{thm:hkpv}. It follows from the fact that these variables are independent with uniform argument that the only possible non-zero contributions to the sum are when $j_1 = \cdots = j_k = i$. Therefore by Kostlan's theorem \ref{thm:kostlan}, Lemma \ref{lem:gamma} and Lemma \ref{lem:rising-sum}, we conclude that
        \[T_{k,M}(N) = \sum_{i=1}^N\E{|z_i|^{2k}} = \sum_{i=1}^N \frac{\G(kM+i)}{\G(i)} = \frac{1}{kM+1}\Fp{kM+1}{N}.\]
    \end{proof}

\subsection{Case $k=2$}
    We now consider bipartite pairings of the edges of a $2M$-gon with two separate $M$-gons, as represented on Figure \ref{fig:case_M_2}. We observe that in this case the polynomial $T_{2,M}$ can be determined directly from proposition \ref{prop:kM_1M}, together with a parity argument.

    \begin{figure}[h!]
        \centering
        \includegraphics[width=6cm]{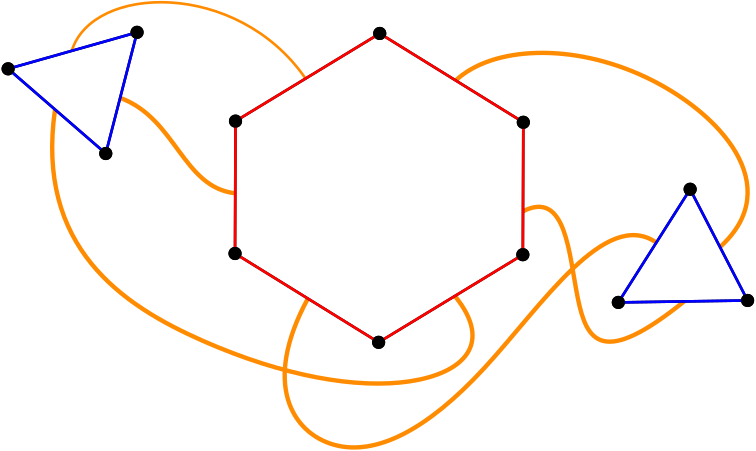}
        \caption{Example of a bipartite pairing of the sides of a hexagon and  two triangles i.e. $k=2$, $M=3$.}
        \label{fig:case_M_2}
    \end{figure}

    \begin{proposition}\label{prop:-MM}
        For $-M \leq N \leq M$ we have the following equality
        \[T_{2,M}(N) = \frac{1}{2M+1}\left[F_{2M+1}^+(N)+\Fm{2M+1}{N}\right]
        = \frac{1}{2M+1}\left[F_{2M+1}^+(N)+\Fp{2M+1}{N-2M}\right].\]
    \end{proposition}
    
    \begin{proof} The values $T_{2,M}(N)$ for $N \in [\![0,M]\!]$ are given by Proposition \ref{prop:kM_1M} and match with the claim.
        As $\deg(T_{2,M}) = 2M-1$ and $T_{2,M}$ is odd from Proposition \ref{prop:TkM}, then for $-M \leq N \leq -1$:
        \[(-1)^{2M-1} T_{2,M}(N) = T_{2,M}(-N) \stackrel{\text{Prop. \ref{prop:kM_1M}}}{=} \frac{1}{2M+1}\Fp{2M+1}{-N} = \frac{-1}{2M+1}\Fm{2M+1}{N}.\]
        Hence $T_{2,M}(N) = \frac{1}{2M+1}\Fm{2M+1}{N}$ for $-M \leq N \leq -1$. We have thus verified that
        \[T_{2,M}(N) = \frac{1}{2M+1}\left[F_{2M+1}^+(N)+\Fm{2M+1}{N}\right]\]
        holds for $-M \leq N \leq M$.
    \end{proof}
    
    Note that the polynomial $T_{2,M}$ has degree $2M-1$ and that by Proposition \ref{prop:-MM} we can compute $2M+1$ of its values. Thus, we can identify $T_{2,M}$.

    \begin{proposition}
        Let $M$ be a positive integer. Then, for every $N$,
        \begin{align*}
            \T{2}{M}{N} 
        &= \frac{1}{2M+1}\left[\Fp{2M+1}{N} + \Fm{2M+1}{N} - 2\Fp{2M+1}{N-M}\right] \\
        &= \frac{1}{2M+1}\left[\Fp{2M+1}{N} - 2\Fp{2M+1}{N-M} + \Fp{2M+1}{N-2M}\right].
        \end{align*}
    \end{proposition}
    
    \begin{proof}
        From Proposition \ref{prop:-MM} we know that the polynomial
        \[\frac{1}{2M+1}[\Fp{2M+1}{N} + \Fm{2M+1}{N}] - \T{2}{M}{N},\]
         has degree $2M+1$ with zeros at $\{-M,-M+1,\ldots,-1,0,1,\ldots,M\}$. Thus it is equal to
        \[C\cdot (N+M)\cdots(N+1) N (N-1)\cdots(N-M)
        =
        C F_{2M+1}^+ (N-M).\]
        By inspection of the leading coefficient, we deduce $C=\frac{2}{2M+1}$ and conclude.
    \end{proof}

    Note that the parity argument generalizes to all values of $k$, yielding the values of $T_{k,M}$ in the integer interval $[\![-M,M]\!]$.

    \begin{proposition}\label{prop:kM_-MM}
        For any $k \geq 1$ and $-M \leq N \leq M$, we have the following equality
        \begin{align*}
            T_{k,M}(N) 
            &= \frac{1}{kM+1}\left[F_{kM+1}^+(N)+(-1)^k\Fm{kM+1}{N}\right] \\
            &= \frac{1}{kM+1}\left[F_{kM+1}^+(N)+(-1)^k\Fp{kM+1}{N-kM}\right]
        \end{align*}
    \end{proposition}
    \begin{proof}
        As $\deg(T_{k,M}) = k(M-1)+1$ and $T_{k,M}$ is either even or odd depending on $k,M$, then for $-M \leq N \leq -1$:
        \[(-1)^{k(M-1)+1} T_{k,M}(N) = T_{k,M}(-N) \stackrel{\text{Prop. \ref{prop:kM_1M}}}{=} \frac{1}{kM+1}\Fp{kM+1}{-N} = \frac{(-1)^{kM+1}}{kM+1}\Fm{kM+1}{N}.\]
        Hence $T_{k,M}(N) = \frac{(-1)^k}{kM+1}\Fm{kM+1}{N}$ for $-M \leq N \leq -1$. Then we notice that
        \[T_{k,M}(N) = \frac{1}{kM+1}\left[F_{kM+1}^+(N)+(-1)^k\Fm{kM+1}{N}\right]\]
        holds for $-M \leq N \leq M$ by proposition \ref{prop:kM_1M}.
    \end{proof}

\vspace{.2in}

\subsection{Case $k=3$}
    We now consider bipartite pairings of the edges of a $3M$-gon and three smaller $M$-gons, as represented on Figure \ref{fig:case_M_3}.

    \begin{figure}[ht]
        \centering
        \includegraphics[width=9cm]{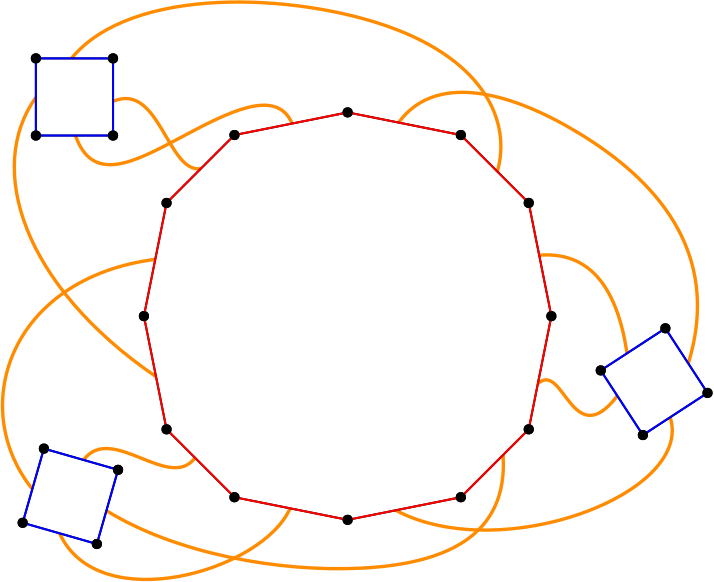}
        \caption{Example of a bipartite pairing of the sides of a 12-gon and  three squares i.e. $k=3$, $M=4$.}
        \label{fig:case_M_3}
    \end{figure}

    The main motivation of presenting the cases $k=3,4$ in detail is to showcase the core idea of the proof of the main result, which relies on Theorem \ref{thm:blocks} on the distribution of low powers of Ginibre eigenvalues. Using this theorem, we have the following.
 
    \begin{proposition}\label{prop:M2M}
        For $M < N \leq 2M$ we have the following equality
        \[T_{3,M}(N) = \frac{1}{3M+1}\Fp{3M+1}{N} - \frac{3}{3M+1} \Fp{3M+1}{N-M}\]
    \end{proposition}

    \begin{proof}
        For $M < N \leq 2M$, from Theorem \ref{thm:blocks} we have $|I_p| \in\{1,2\}$ for all $1 \leq p \leq M$. We also have
        \[T_{3,M}(N) = \E{\sum_{i,j_1,j_2,j_3 = 1}^N \l_i^{kM} \bar{\l_{j_1}^M}\bar{\l_{j_2}^M}\bar{\l_{j_3}^M}} = \E{\sum_{i,j_1,j_2,j_3 = 1}^N z_i^3\bar{z_{j_1}}\bar{z_{j_2}}\bar{z_{j_3}}},\]
        still with the notations of Theorem \ref{thm:blocks}. Now we consider two possible cases: either the index $i$ is in a block of size one, or it is in a block of size two. \\
        
    \noindent \textbf{First case: $i \in I_p$ where $|I_p| = 1$, i.e. $I_p = \{p\}$.} Then $z_i$ is independent from other $z_j$'s, thus we only need to consider expressions of the type
        \[\E{z_i^k\bar{z_i}^m} = \frac{1}{Z_{N,M,p}}\int_{\C} z_i^k \bar{z_i}^m |z_i|^{2(p-M)/M}e^{-|z_i|^{2/M}} \dd m(z_i).\]
        It is non-zero only if $k=m$, because $|z_i|^{2(p-M)/M}e^{-|z_i|^{2/M}}$ is rotational invariant. Therefore the non-zero contribution we have in this case is when $i = j_1=j_2=j_3$ with
        \begin{align*}
            \E{|z_i|^{6}} 
            &= \frac{1}{Z_{N,M,p}} \int_\C |z_i|^{6 + 2(p-M)/M}e^{-|z_i|^{2/M}} \dd m(z_i) \\
            &= \frac{2\pi}{\pi\cdot1!\cdot M\cdot(p-1)!} \int_{0}^\infty r^{6+1+2(p-M)/M} e^{-r^{2/M}} \dd r \\
            &= \frac{2}{M(p-1)!} \cdot \frac{M}{2}\int_0^\infty u^{3M + p-1}e^{-u} \dd u \\
            &= \frac{\G(3M+p)}{\G(p)}.
        \end{align*}

    \noindent \textbf{Second case: $i \in I_p$ where $|I_p| = 2$, i.e. $I_p =\{p,p+M\}$.} Let $j$ be the second element in $I_p$ different from $i$. Then again we find that other $z_r$'s are independent, thus we also need to consider expressions such as $\E{z_i^k \bar{z_i}^{m}\bar{z_j}^l}$. More generally, we consider
        \[\E{z_i^a \bar{z_i}^b z_j^c \bar{z_j}^d} = \frac{1}{Z_{N,M,k}}\int \int z_i^a \bar{z_i}^b z_j^c \bar{z_j}^d |z_i-z_j|^2 |z_iz_j|^{2(p-M)/M} e^{-|z_i|^{2/M} - |z_j|^{2/M}} \dd m(z_i) \dd m(z_j)\]
        Note that integration is taken with respect to the density of the block $I_p$, given in Theorem~\ref{thm:blocks}. This becomes:
        \[\frac{1}{Z_{N,M,k}}\int \int z_i^a \bar{z_i}^b z_j^c \bar{z_j}^d (|z_i|^2 + |z_j|^2 - z_i\bar{z_j} - \bar{z_i}z_j) |z_iz_j|^{2(p-M)/M} e^{-|z_i|^{2/M} - |z_j|^{2/M}} \dd m(z_i) \dd m(z_j).\]
        Hence, by inspection we get a non-zero contribution in the following cases:
        \begin{itemize}
            \item $a=b$ and $c=d$ (contributions associated to $|z_i|^2$ and $|z_j|^2$)
            \item $a+1=b$ and $c=d+1$ (contribution associated to $z_i\bar{z_j}$)
            \item $a=b+1$ and $c+1=d$ (contribution associated to $\bar{z_i}z_j$)
        \end{itemize}
        This once again due to the rotational invariance of the remaining part. In our context, this boils down to two possibilities for $\E{z_i^k \bar{z_i}^{m}\bar{z_j}^l}$, which we refer to as diagonal and off-diagonal contributions. \\
        
    \noindent \textbf{-- Diagonal contributions;} the case $a=b$, $c=d$ occurs if $m = k$ and $l=0$. This corresponds to
        \begin{align*}
            \E{|z_i|^{6}} 
            &= \frac{1}{Z_{N,M,p}} \int_{z_i}\int_{z_j} |z_i|^{6}(|z_i|^ 2+|z_j|^2)|z_iz_j|^{2(p-M)/M}e^{-|z_i|^{2/M} - |z_j|^{2/M}} \dd m(z_i) \dd m(z_j) \\
            &= \frac{1}{Z_{N,M,p}} \int_{z_i}\int_{z_j} |z_i|^{6+2+2(p-M)/M}|z_j|^{2(p-M)/M}e^{-|z_i|^{2/M} - |z_j|^{2/M}} \dd m(z_i) \dd m(z_j) \\
            &+ \frac{1}{Z_{N,M,p}} \int_{z_i}\int_{z_j} |z_i|^{6+2(p-M)/M}|z_j|^{2+2(p-M)/M}e^{-|z_i|^{2/M} - |z_j|^{2/M}} \dd m(z_i) \dd m(z_j) \\
            &= \frac{\pi^2M^2}{Z_{N,M,p}}\left[\G(4M+p)\G(p) + \G(3M+p)\G(M+p)\right] \\
            &= \frac{1}{2}\left[\frac{\G(3M+p+M)}{\G(p+M)} + \frac{\G(3M+p)}{\G(p)}\right],
        \end{align*}
        which is also the same as $\E{|z_j|^{6}}$. \\

    \noindent \textbf{-- Off-diagonal contributions;} the $a=b+1$ and $c+1=d$ occurs if $m = k-1$ and $l=1$. This can only correspond to
        \begin{align*}
            \E{|z_i|^{4} z_i \bar{z_j}}
            &= \frac{-1}{Z_{N,M,p}}\int_{z_i}\int_{z_j} |z_i|^{4} z_i \bar{z_j}\cdot(\bar{z_i}z_j) \cdot|z_iz_j|^{2(p-M)/M} e^{-|z_i|^{2/M}-|z_j|^{2/M}} \dd m(z_i) \dd m(z_j)\\
            &= \frac{-1}{Z_{N,M,p}}\int_{z_i}\int_{z_j} |z_i|^{4+2(k-M)/M} |z_j|^{2+2(p-M)/M} e^{-|z_i|^{2/M}-|z_j|^{2/M}} \dd m(z_i) \dd m(z_j) \\
            &= \frac{-1}{Z_{N,M,p}}\int_{z_i}|z_i|^{4+2(p-M)/M} e^{-|z_i|^{2/M}}\dd m(z_i) \int_{z_j}  |z_j|^{2+2(p-M)/M} e^{-|z_j|^{2/M}}  \dd m(z_j) \\
            &= \frac{-\pi^2M^2}{\pi^2\cdot2!\cdot M^2 \G(p)\G(p+M)}\cdot\G(3M+p)\G(M+p) \\
            &= -\frac{1}{2}\cdot\frac{\G(3M+p)}{\G(p)}
        \end{align*}

        To finish the calculation for the second case, we also need to count the number of occurrences of the expression $\E{|z_i|^{4} z_i \bar{z_j}}$ can be found in the sum of $\E{z_i^3 \overline{z_{j_1} z_{j_2} z_{j_3}}}$ over all $i, j_1, j_2, j_3$. We notice that from $\{j_1,j_2,j_3\}$ two of them must equal to $i$ and the other to $j$. Thus we have $3$ occurrences of a term of type $\E{|z_i|^{4} z_i \bar{z_j}}$. Moreover observe that $p$ will range from $1$ to $N-M$, as there are $N-M$ sets $I_p$ of size $2$. Hence, combining both cases we achieve:
        \begin{align*}
        T_{3,M}(N)
        &= \sum_{\substack{i\in I_p\\ |I_p| = 1}} \E{|z_i|^{6}} + \sum_{\substack{i\in I_p \\ |I_p| = 2}} \E{|z_i|^6} - 3\sum_{p=1}^{N-M}\frac{\G(3M+p)}{\G(p)} \\
        &= \sum_{p=1}^N\frac{\G(3M+p)}{\G(p)} - 3\sum_{p=1}^{N-M}\frac{\G(3M+p)}{\G(p)}\\
        &= \frac{1}{3M+1}\Fp{3M+1}{N} - \frac{3}{3M+1} \Fp{3M+1}{N-M}, & \text{(by Lemma \ref{lem:rising-sum})}
        \end{align*}
        which was the claim.
    \end{proof}
    Now we can attain the following:
    \begin{proposition}
        Let $M$ be a positive integer. Then we have the following
        \[T_{3,M}(N) = \frac{1}{3M+1}\left[\Fp{3M+1}{N} -3\Fp{3M+1}{N-M}+ 3\Fp{3M+1}{N-2M} -\Fp{3M+1}{N-3M}\right]\]
        for every integer $N$.
    \end{proposition}

    \begin{proof}
        From proposition \ref{prop:kM_-MM} we have
        \[\T{3}{M}{N} = \frac{1}{3M+1}\left[\Fp{3M+1}{N} - \Fm{3M+1}{N}\right] ~~ \text{for} ~ -M \leq N \leq M.\]
        By proposition \ref{prop:M2M} we get
        \[T_{3,M}(N) = \frac{1}{3M+1}\left[\Fp{3M+1}{N} - 3\Fp{3M+1}{N-M}\right] ~~ \text{for} ~~ M < N \leq 2M.\]
        
        Therefore by considering the polynomial
        \[\frac{1}{3M+1}\left[\Fp{3M+1}{N} -3\Fp{3M+1}{N-M}-\Fp{3M+1}{N-3M}\right] - T_{3,M}(N)\]
        we get $3M+1$ zeros: $\{-M,-M+1\ldots,2M\}$. As this polynomial is of degree $3M+1$ we explicitly write it as
        \[C(N+M)(N+M-1)\ldots(N-2M) = C \cdot \Fp{3M+1}{N-2M}.\]
        To determine the constant $C$, we note that $\deg(T_{3,M}) = 3M-2$, thus by inspection of the leading coefficient (of degree ${3M+1}$), we find $C = -3/(3M+1)$.
    \end{proof}

    Yet again, we are able to specialize the above proposition to calculate exactly the values of $T_{k,M}$ in the integer interval $[\![M,2M]\!]$, for any value of $k$.

    \begin{proposition}\label{prop:kM_M2M}
        For any $k \geq 1$ and $M < N \leq 2M$ we have the following equality
        \[T_{k,M}(N) = \frac{1}{kM+1}\Fp{kM+1}{N} - \frac{k}{kM+1} \Fp{kM+1}{N-M}\]
    \end{proposition}

    The proof is entirely analogous to the one of Proposition \ref{prop:M2M}. We do not repeat it for the sake of brevity. Note that it is also a special case of Theorem \ref{thm_bipartite_HZ_gen_fun}, which will be proved in Section \ref{sec_main_result}.

\subsection{Case $k=4$}
    The last case for which we give a specific computation is $k=4$, which means that we consider bipartite pairings between a $4M$-gon and four $M$-gons, as represented on Figure \ref{fig:case_M_4}. Based on the moments computations of the previous case $k=3$, and the same symmetry trick used for case $k=2$, we establish the following identity:
    
    \begin{proposition}
        Let $M$ be a positive integer. Then for every $N$,
        \[T_{4,M}(N) = \frac{1}{4M+1} \sum_{l=0}^4 (-1)^ l \binom{4}{l}\Fp{4M+1}{N-lM}.\]
    \end{proposition}

        \begin{figure}[h]
        \centering
        \includegraphics[width=9cm]{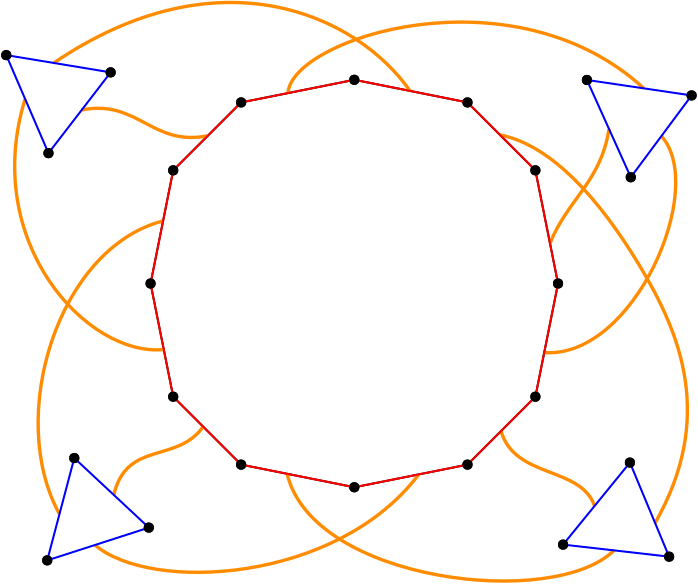}
        \caption{Example of a bipartite pairing of the sides of a 12-gon and  four triangles i.e. $k=4$, $M=3$.}
        \label{fig:case_M_4}
    \end{figure}
    
    \begin{proof}
        Using Proposition \ref{prop:kM_-MM} we get:
        \[\T{4}{M}{N} = \frac{1}{4M+1}\left[\Fp{4M+1}{N} + \Fm{4M+1}{N}\right] ~~ \text{for} ~ -M \leq N \leq M\]
        and by Proposition \ref{prop:kM_M2M}
        \[T_{4,M}(N) = \frac{1}{4M+1}\left[\Fp{4M+1}{N} - 4\Fp{4M+1}{N-M}\right] ~~ \text{for} ~~ M < N \leq 2M.\]
        By the same symmetry trick we get for $-2M \leq N < -M$:
        \[-T_{4,M}(N) = T_{4,M}(-N) = \frac{-1}{4M+1}\left[\Fm{4M+1}{N} - 4\Fm{4M+1}{N+M}\right],\]
        thus
        \[T_{4,M}(N) = \frac{1}{4M+1}\left[\Fp{4M+1}{N-4M} - 4\Fp{4M+1}{N-3M}\right].\]
        Therefore by considering the polynomial
        \[\frac{1}{4M+1}\left[\Fp{4M+1}{N} -4\Fp{4M+1}{N-M}-4\Fp{4M+1}{N-3M} + \Fp{4M+1}{N-4M}\right] - T_{4,M}(N),\]
        we get its $4M+1$ zeros: $\{-2M,\ldots,2M\}$, as it has degree $4M+1$ we conclude that it has the form
        \[C(N-2M)(N-2M+1)\cdots(N+2M) = C F_{4M+1}^+ (N-2M).\]
        By inspection of the leading coefficient, we obtain $C= \frac{6}{4M+1}$ and conclude.
    \end{proof}

    This ends the verification of the first cases: we have so far verified the validity of our main result (Theorem \ref{thm_bipartite_HZ_gen_fun}) for $k=1,2,3,4$ by \textit{ad hoc} computations, identifying the polynomial $T_{k,M}$ on specific integer intervals. In the next section, we establish the general result for every $k \geq 1$, by computing $T_{k,M}(N)$ at some other specific values.

\section{Bipartite formula: the general case}\label{sec_main_result}
\subsection{The proof of the formula} Let us now restate our main theorem and provide a proof of it.

    \begin{theorem}[Bipartite generating function]
        Let $k$ and $M$ be positive integers. Then, for every $N \geq 0$,
        \[T_{k,M}(N) = \frac{1}{kM+1} \sum_{l=0}^k (-1)^l\binom{k}{l}\Fp{kM+1}{N-lM}.\]
    \end{theorem}

    \begin{proof}
        We first note that both sides of the equality are polynomials in $N$, therefore it is enough to prove for $N = qM$ for $q > k$ -- which is an infinite set of values on which the block decomposition from Theorem \ref{thm:blocks}  will be particularly tractable. Let us recall this decomposition: we have that
        \[\{\l_1^M,\ldots,\l_N^M\} \stackrel{d}{=} \{z_1,\ldots,z_N\},\]
        with $M$ independent blocks of size $q$ for parameters $p \in [\![1,M]\!]$ with joint density
        \[\frac{1}{Z_{N,M,p}} \prod_{1\leq i < j \leq q} |z_i-z_j|^2 \prod_{i=1}^q |z_i|^{\frac{2(p-M)}{M}}e^{-|z_i|^{2/M}}\dd m(z_i),\]
        where
        \[Z_{N,M,p} = \pi^{q}q! M^{q} \prod_{j=0}^{q-1} (jM+p-1)!.\]
        We want to calculate, for $N=qM$ specifically,
        \begin{align*}
        T_{k,M}(N) = \E{\tr(G^{kM}) \bar{\tr(G^M)}^k} &= \sum_{i,j_1,\ldots,j_k=1}^n \E{\l_i^{kM} \bar{\l_{j_1}}^M\ldots\bar{\l_{j_k}}^M}\\
        & = \sum_{i,j_1,\ldots,j_k=1}^q \E{z_i^k \bar{z_{j_1}}\bar{z_{j_2}}\cdots\bar{z_{j_k}}}.
        \end{align*}
        We note that, by inspection, non-zero contributions can occur only when all $z_j$ variables are in the same block as $z_i$. Let us compute the contribution for the block with index $p$. We assume that $I_p = \{z_1,\ldots,z_q\}$ (without loss of generality, as all blocks have same size $q$, and by exchangeability the indices do not matter to the computation).
        \[ \sum_{i,j_1,\ldots,j_k=1}^q \E{z_i^k \bar{z_{j_1}}\bar{z_{j_2}}\cdots\bar{z_{j_k}}} = q\cdot \sum_{\a_1 + \ldots + \a_{q} = k} \frac{k!}{\a_1 ! \cdots \a_q!} \E{z_1^k \bar{z_1}^{\a_1}\cdots\bar{z_q}^{\a_q}}\]
        This comes from considering how many times ($\alpha_{\ell}$) the index $\ell$ appears among $j_1, \dots, j_k$. The factor $q$ comes from fixing $i=1$ instead of summing over $i$, and the factor $k!/(\a_1!\cdots\a_q!)$ is the number of ways in which the weights $\alpha_{\ell}$ can be realized. We then write down the expectation corresponding to the block $I_p$, which yields
        \[ \frac{q}{Z_{N,M,p}} \sum_{\a_1+\ldots+\a_q = k}\frac{k!}{\a_1! \cdots \a_q!}\int_{\C^q} z_1^k\bar{z_1}^{\a_1}\cdots\bar{z_q}^{\a_q} \prod_{i<j} |z_i - z_j|^2 \prod_{i=1}^q |z_i|^{\frac{2(p-M)}{M}}e^{-|z_i|^{2/M}}\dd m(z_i).\]
        Observe that $\displaystyle \prod_{i=1}^q |z_i|^{\frac{2(p-M)}{M}}e^{-|z_i|^{2/M}}$ is rotational invariant, hence we focus on the quantity
        \beq\label{eq_sum_permutations} z_1^k\bar{z_1}^{\a_1}\cdots\bar{z_q}^{\a_q} \prod_{i<j} |z_i - z_j|^2
        = \sum_{\s\in S_q} \sum_{\t \in S_q} \sign(\s)\sign(\t) z_1^k\bar{z_1}^{\a_1}\cdots\bar{z_q}^{\a_q} \prod_{i=1}^{q} z_i^{\s(i)-1} \bar{z_i}^{\t(i)-1}.
        \eeq
        When integrating the above expression against a rotationally invariant weight, a non-zero contribution can only occur if the following equalities hold:
        \[
        \begin{cases}
            k + \s(1) = \a_1 + \t(1), \\
            \s(i) = \a_i + \t(i) & \text{for} ~~ i \geq2.
        \end{cases}
        \]
        Note that the only index $i$ for which $\s(i) - \t(i)$ could be negative is for $i=1$, elsewhere it has to be non-negative, as $\a_i \geq 0$. This imposes a strong constraint on the structure of $\s$ and $\t$, which we will exploit; in the following argument, we consider the permutation $\pi = \t \circ \s^{-1}$ and denote $\sigma(1) = a$ and $\tau(1) = a+r$, where $r = k-\a_1 \geq 0$.
        \vspace{0.5em}
        
        \noindent \textbf{First case: $r=0$.} We first consider the case when $r = 0$, which forces $\a_1 = k$ and $\a_i = 0$ for all $i\geq2$, as $k = \a_1 + \cdots \a_q$ with $\a_{\ell} \geq 0$ for every $\ell$.
        Thus, for a nonzero contribution we must have $\s = \t$, and we are left with integrating
        \[\frac{q}{Z_{N,M,p}}\sum_{\s \in S_q} |z_1|^{2(k+\s(1)-1)} \prod_{i=2}^q |z_i|^{2(\sigma(i)-1)} \prod_{i=1}^q |z_i|^{\frac{2(p-M)}{M}}e^{-|z_i|^{2/M}}\dd m(z_i).\]
        By a direct change of variable, we have that
        \beq \int_\C |z|^{a+2(p-M)/M}e^{-|z|^{2/M}} \dd m(z) = \pi M\G(aM/2 + p) \eeq
        Substituting the above identity, we obtain the following expression for these contributions:
        \[\frac{qM^q\pi^q}{Z_{N,M,p}} \sum_{\s \in S_q}\frac{\G((k+\s(1)-1)M + p)}{\G((\s(1)-1)M + p)} \prod_{j=0}^{q-1} \G(jM+p) = \frac{1}{(q-1)!} \sum_{\s \in S_q}\frac{\G((k+\s(1)-1)M + p)}{\G((\s(1)-1)M + p)}.\]
        Now we notice that for any $j \in [n]$, we have exactly $(q-1)!$ permutations such that $\s(1) = j$. Therefore the sum of non zero contributions corresponding to $r=0$ simplifies to
        \beq\label{eq_r0_contribution} \sum_{j=0}^{q-1} \frac{\G((j+k)M + p)}{\G(jM+p)}.\eeq

        \noindent \textbf{Second case: $r\geq 1$.} Now assume that $r = k - \a_1 \geq 1$. The key remark is then that for the contribution to be nonzero, $\pi = \t \circ \s^{-1}$ must have a specific structure, namely, it has only one cycle that is not a fixed point. Let us explore the structure of $\pi$ in some detail before carrying on with the computation.
        \[
        \begin{matrix}
        1 & 2 & \ldots & a & \ldots & a+r & \ldots & q-1 & q\\
        1 & 2 & \ldots& a+r & - & - & \ldots & q-1 & q\\
        \end{matrix}
        \]
        Recall that $a = \sigma(1)$. We first notice that, as
        $$
        \forall i \neq 1, \quad \s(i) - \t(i) \geq 0,
        $$
        we must have, by setting $j=\s (i)$ and using the fact that $\s$ is a bijection,
        $$
        \forall j \neq a, \quad  j - \pi(i) \geq 0.
        $$
        As for $j=a$, we have that 
        $$a - \pi(a) = \s (a) - \t (1) = -r  < 0. $$
        These constraints mean that for $\pi$ we are restricted to only one \textit{right} jump at $i = a$ and the rest must be \textit{left} jumps, which is only possible in one cycle as $\pi$ must also be a bijection. By inspection, this cycle must also have $a$ and $a+r$ as extremal indices. More precisely, we must have some sequence $0=\b_0 < \b_1 < \ldots < \b_m < \b_{m+1} = r$, such that
        \[\pi(a+\b_{i+1}) = a + \b_i, \]
        as represented on Figure \ref{fig:perm}. All other indices are fixed points of $\pi$. \\
        
        \begin{figure}[h!]
            \centering
            \includegraphics[width=9cm]{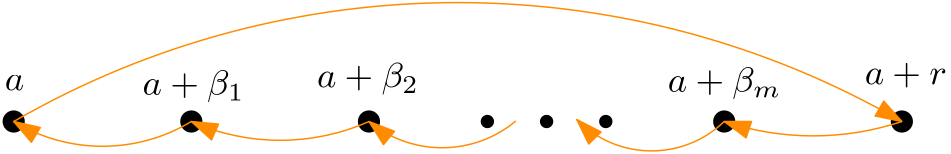}
            \caption{A visualization of the structure of $\pi =  \tau \circ \sigma^{-1}$, with a unique right jump from $a$ to $a+r$ and $m+1$ left jumps from $a+\b_{i+1}$ to $a+\b_i$.}
            \label{fig:perm}
        \end{figure}
        
        \noindent We notice that if we set $j = \sigma^{-1}(a+\b_{i+1})$, then we have 
        \[\s(j) - \t(j) = a + \b_{i+1} - (a+\b_i) = \b_{i+1} - \b_i\]
        So we have that $\a_j$ is either $0$ or some $\b_{i+1}- \b_i$ for some $i$. And 
        \[\a_1 = \s(1) - \t(1) +k = k - r.\]
        With this construction, the datum of a permutation $\sigma$ and of the tuple $(\b_1,\ldots,\b_m)$ uniquely determines $\tau$, as we then have
        \[\t = ( a+r, a+\b_m, \ldots , a+\b_1 , a ) \circ \s.\]
        In particular, note that $\sign(\t) = (-1)^{m+1} \cdot \sign(\s)$. $\sigma$ can be any permutation, with the only requirement that $\s(1) \leq q-r$ to accommodate the jump, and $r$ would range from $1$ to $k$, as this is the largest jump we can make from $\s(1)$. Therefore we can rewrite the sum over permutations \eqref{eq_sum_permutations} as follows, for a given parameter $k \geq r \geq 1$:
        \beq\label{eqn:longperm} 
        \sum_{m=0}^{r-1} \ \sum_{ 0 < \b_1 < \ldots < \b_m < r } \sum_{\substack{\s \in S_q \\ \s(1) \leq q-r}} \frac{k! (-1)^{m+1} }{\prod_{i=0}^m (\b_{i+1} - \b_i)! \cdot (k-r)!} |z_1|^{2(k+\s(1)-1)} \prod_{i=2}^q |z_1|^{2(\s(i)-1)}.
        \eeq
        We first focus on the combinatorial expression
        \[\sum_{0 < \b_1 <\ldots < \b_m < r} \frac{1}{\prod_{i=0}^m (\b_{i+1} - \b_i)!}, \]
        which we simplify in the following way. By rewriting $a_i = \b_{i+1} - \b_i$ with convention $\b_0~=~0$ and $\b_{m+1} = r$, we have $a_0 + \ldots + a_m = r$ and all $a_i \geq 1$. Therefore, we have
        \[ \sum_{0 < \b_1 <\ldots < \b_m < r} \frac{1}{\prod_{i=0}^m (\b_{i+1} - \b_i)!}= \frac{1}{r!}\sum_{\substack{a_0+\ldots+a_m = r \\ a_i \geq 1}} \binom{r}{a_0,\ldots,a_m}.\]
        The sum represent exactly the number of ways to partition $r$ objects into $m+1$ non-empty sets, together with a labeling of these sets from $1$ to $m+1$. Therefore this sum is equal to $(m+1)! \cdot \ssec{r}{m+1}$ by definition of Stirling numbers, and hence
        \[\sum_{0 < \b_1 <\ldots < \b_m < r} \frac{1}{\prod_{i=0}^m (\b_{i+1} - \b_i)!} = \frac{(m+1)!}{r!} \ssec{r}{m+1}.\]
        We now perform the sum over $m$, and find the following quantity
        \[\sum_{m=0}^{r-1} (-1)^{m+1}(m+1)!\ssec{r}{m+1},\]
        which we can simplify by using Lemma \ref{lem:stirling_second_power}; moreover, as $\ssec{r}{0} = 0$, we find that the sum is exactly $(-1)^r$. Hence
        \[\sum_{m=0}^{r-1}(-1)^{m+1}\sum_{0 < \b_1 <\ldots < \b_m < r} \frac{1}{\prod_{i=0}^m (\b_{i+1} - \b_i)!} = \frac{(-1)^r}{r!}.\]
        Therefore, the quantity \eqref{eqn:longperm} is equal to
        \begin{align*}
        \frac{q}{Z_{N,M,p}}\sum_{r=1}^k (-1)^r\binom{k}{r} & \sum_{\substack{\s \in S_q \\ \s(1) \leq q-r}} \int_{\C^q}|z_1|^{2(k+\s(1)-1)} \prod_{i=2}^q |z_1|^{2(\s(i)-1)} \prod_{i=1}^q |z_i|^{\frac{2(p-M)}{M}}e^{-|z_i|^{2/M}}\dd m(z_i)\\
        &=\frac{qM^q\pi^q}{Z_{N,M,p}}\sum_{r=1}^k(-1)^r\binom{k}{r} \sum_{\substack{\s \in S_q \\ \s(1) \leq q-r}}\frac{\G((k+\s(1)-1)M+p)}{\G((\s(1)-1)M + p)} \prod_{j=0}^{q-1}\G(jM+p). \\
        &=\sum_{r=1}^k (-1)^r\binom{k}{r} \sum_{j=0}^{q-r-1}\frac{\G((j+k)M + p)}{\G(jM+p)}.
        \end{align*}
        Therefore, by summing over all possible block indices $p=1,\ldots,M$ and all values of $r=0, \dots, k$ including the contribution \eqref{eq_r0_contribution}, we find the following expression for the value $T_{k,M}(N)$ when $N=qM$:
        \begin{align*}
        \sum_{p=1}^{M} \sum_{r=0}^k (-1)^r \binom{k}{r} &\sum_{j=0}^{q-r-1}\frac{\G((j+k)M + p)}{\G(jM+p)} 
         =  \sum_{r=0}^k (-1)^r \binom{k}{r} \sum_{j=0}^{q-r-1} \sum_{p=1}^{M}\frac{\G((j+k)M + p)}{\G(jM+p)} \\
        &= \frac{1}{kM+1}\sum_{r=0}^k (-1)^r \binom{k}{r} \sum_{j=0}^{q-r-1} F_{kM+1}^+((j+1)M) - F_{kM+1}^+(jM) \\
        &= \frac{1}{kM+1}\sum_{r=0}^k (-1)^r \binom{k}{r} F_{kM+1}^+((q-r)M) \\
        &= \frac{1}{kM+1}\sum_{r=0}^k (-1)^r \binom{k}{r} F_{kM+1}^+(N -rM).
        \end{align*}
        which concludes the proof, as these values fully characterize the polynomial.
    \end{proof}

    \subsection[Recovering the coefficient]{Recovering the coefficient $\e_g(k,M)$}\label{sec:reccoef}
    We noticed that we are able to extract the exact coefficient from the equality
    \[\sum_{g=0}^{\lfloor k(M-1)/2\rfloor} \e_g(k,M) N^{k(M-1)+1-2g} = \frac{1}{kM+1} \sum_{l=0}^k (-1)^l \binom{k}{l} F_{kM+1}^+(N-lM).\]
    We expand the rising factorial in terms of Stirling numbers of the first kind as follows
    \[F_{kM+1}^+(N-lM) = \sum_{p=0}^{kM+1}\sfir{kM+1}{p}(N-lM)^p.\]
    Hence
    \begin{align*}
        \frac{1}{kM+1} \sum_{l=0}^k (-1)^l \binom{k}{l} F_{kM+1}^+ & (N-lM)  
        = \frac{1}{kM+1} \sum_{l=0}^k \sum_{p=0}^{kM+1} (-1)^l \binom{k}{l} \sfir{kM+1}{p}(N-lM)^p \\
        &= \frac{1}{kM+1} \sum_{l=0}^k \sum_{p=0}^{kM+1} (-1)^l \binom{k}{l} \sfir{kM+1}{p} \sum_{r=0}^p \binom{p}{r} N^r (-lM)^{p-r} \\
        &= \frac{1}{kM+1}   \sum_{p=0} ^{kM+1} \sum_{r=0}^p \sfir{kM+1}{p}  \binom{p}{r} N^r (-M)^{p-r} \sum_{l=0}^k(-1)^l \binom{k}{l} l^{p-r}
    \end{align*}
    We use the closed form for the Stirling number of the second kind, Proposition \ref{prop:stirling_second_closed}
    \[ \sum_{l=0}^k(-1)^l\binom{k}{l} l^{p-r} = (-1)^k k!\ssec{p-r}{k}\]
    Hence
    \begin{align*}
        &= \frac{1}{kM+1}   \sum_{p=0} ^{kM+1} \sum_{r=0}^p \sfir{kM+1}{p}  \binom{p}{r} N^r (-M)^{p-r} (-1)^k k!\ssec{p-r}{k} \\
        &= \frac{1}{kM+1}  \sum_{r=0}^{kM+1} \sum_{p=r} ^{kM+1} \sfir{kM+1}{p}  \binom{p}{r} N^r (-M)^{p-r} (-1)^k k!\ssec{p-r}{k} \\
        &= \frac{(-1)^k k!}{kM+1}  \sum_{r=0}^{k(M-1)+1} \left[\sum_{p=k} ^{kM+1-r} \sfir{kM+1}{p+r}  \binom{p+r}{r} \ssec{p}{k} (-M)^{p}\right] N^r
    \end{align*}
    Finally we get
    \[\e_g(k,M) = \frac{(-1)^k k!}{kM+1}\sum_{p=k}^{k+2g} \sfir{kM+1}{p+k(M-1)+1-2g}\binom{p+k(M-1)+1-2g}{p}\ssec{p}{k}(-M)^p\]
    Reindexing the sum, we get:
    \begin{proposition}
        For all integers $k,M>0$, $g \geq 0$, we have the following identity
        \[\e_g(k,M) = \frac{(-1)^k k!}{kM+1}\sum_{p=0}^{2g} \sfir{kM+1}{p+kM+1-2g}\binom{p+kM+1-2g}{p+k}\ssec{p+k}{k}(-M)^{p+k}.\]
    \end{proposition}

\noindent In particular, we notice that the leading coefficient: $\e_0(k,M)$ simplifies to 
    \beq\label{Fuss_Catalan}
    \e_0(k,M) = \frac{(-1)^kk!}{kM+1} \sfir{kM+1}{kM+1}\binom{kM+1}{k}\ssec{k}{k}(-M)^{k} = \frac{k! \cdot M^k}{kM+1} \binom{kM+1}{k},
    \eeq
which one can identify as a Fuss-Catalan number, as expected.

\section{Epilogue: recovering the original Harer-Zagier Formula}\label{sec_recovering_HZ}

    In this last section, we explain why our bipartite theorem actually generalizes the original Harer-Zagier formula: indeed, one can recover Theorem~\ref{thm_HZ_gen_fun} directly from Theorem~\ref{thm_bipartite_HZ_gen_fun}. For this, the key is to notice the following link between the original Harer-Zagier coefficients $\e_g(n)$ and the bipartite coefficients $\e_g(k,M)$, at the specific values $M=2$ and $k=n$.

    \begin{proposition}\label{prop:correspondence}
    $\e_g(k,2) = 2^k \cdot k! \cdot \e_g(k).$
    \end{proposition}
    
    \begin{proof}
    By definition, $\e_g(k,2)$ counts bipartite pairings between one $2k$-gon and $k$ separate $2$-gons, as illustrated on Figure~\ref{fig:recover}. Notice, then, that when two edges of the $2k$-gon are joined to the same $2$-gon, the surface obtained is the same as with a direct edge pairing of the two edges of the $2k$-gon, with no $2$-gon involved. Therefore, every bipartite pairing of a $2k$-gon and $k$ distinct $2$-gons can be naturally associated to a pairing of edges of a single $2k$-gon, yielding the same surface -- which is precisely one of the pairings counted by the coefficients $\e_g(k)$. This, however, is not a one-to-one correspondence, as a given pairing of the edges of the $2k$-gon can be obtained in exactly $k! \cdot 2^k$ ways, as for each pairing of two edges, we can choose 1) which $2$-gons we use, out of $k$ possible choices, which gives a factor $k!$ and 2) in which direction it is used, with two possible choices for each $2$-gon, giving the factor $2^k$.
    \end{proof}

    \begin{figure}[ht]
        \centering
        \includegraphics[width=6cm]{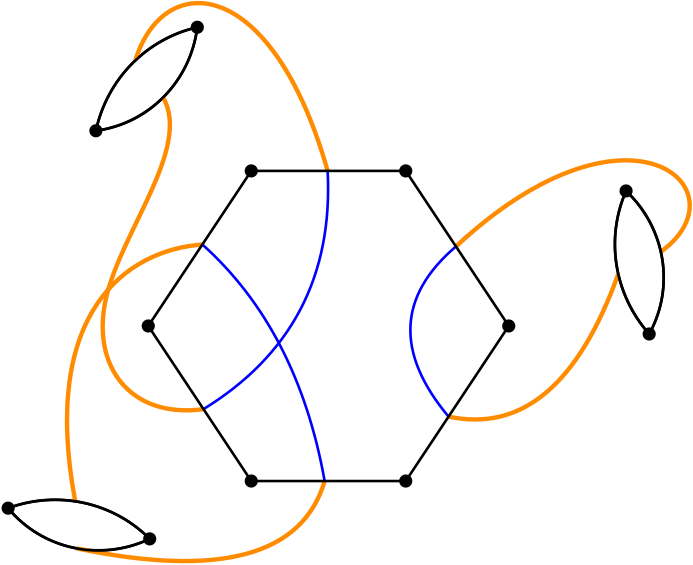}
        \caption{Example of a bipartite pairing of the sides of a hexagon with those of three $2$-gons, i.e. $k=3$, $M=2$, represented in orange. This pairing yields the same surface as a pairing of the edges of the hexagon, represented in blue. This could have been achieved in $3! \cdot 2^3 = 48$ ways.}
        \label{fig:recover}
    \end{figure}

Starting with Theorem~\ref{thm_bipartite_HZ_gen_fun}, the bipartite generating function with $M=2$, $k=n$ is given by the formula:
\beq 
\sum_{g=0}^{\lfloor n/2\rfloor} \e_g(n,2) N^{n+1-2g}
=
\frac{1}{2n+1}\sum_{l=0}^n (-1)^l \binom{n}{l}F_{2n+1}^+(N-2l)
\eeq
Replacing the bipartite coefficients by the Harer-Zagier coeffients according to Proposition~\ref{prop:correspondence}, we find
\beq 
\sum_{g=0}^{\lfloor n/2\rfloor} \e_g(n) N^{n+1-2g}
=
\frac{1}{n! \ 2^n \ (2n+1)}\sum_{l=0}^n (-1)^l \binom{n}{l}F_{2n+1}^+(N-2l)
\eeq
where, by definition,
\beq
F_{2n+1}^+(N-2l) = (2n+1)! \binom{N+2n-2l}{2n+1},
\eeq
so that
\beq 
\sum_{g=0}^{\lfloor n/2\rfloor} \e_g(n) N^{n+1-2g}
=
\frac{(2n)!}{n! \cdot 2^n} \sum_{l=0}^n (-1)^l \binom{n}{l} \binom{N+2n-2l}{2n+1}.
\eeq
This is a closed expression for the Harer-Zagier generating function, but is not yet obviously the same as the one cited in Theorem~\ref{thm_HZ_gen_fun}. We transform this expression using the coefficient extractor operator, denoted by $[x^{m}]$, simply noticing that
    \[\binom{N+2n-2l}{2n+1} = [x^{2n+1}] (1 + x)^{N+2n-2l},\]
and thus
\begin{align*}
\sum_{l=0}^n (-1)^l \binom{n}{l}\binom{N+2n-2l}{2n+1}  & = [x^{2n+1}] \left( (1+x)^N\sum_{l=0}^n (-1)^l \binom{n}{l}(1+x)^{2(n-l)} \right) \\
    & = [x^{2n+1}] \left( (1+x)^N(-1+(1+x)^2)^n \right) \\
    & = [x^{2n+1}] \left( (1+x)^N(2+x)^n x^n \right) \\
    & = [x^{n+1}] \left( (1+x)^N(2+x)^n \right) \\
    & = \sum_{l=0}^n 2^l \binom{N}{l+1} \binom{n}{n-l}.
\end{align*}
We conclude that
\beq 
\sum_{g=0}^{\lfloor n/2\rfloor} \e_g(n) N^{n+1-2g} 
=
(2n-1)!! \ \sum_{l=0}^n 2^l \binom{N}{l+1} \binom{n}{l}
\eeq
which is the same expression as Theorem~\ref{thm_HZ_gen_fun}. Thus, we have recovered the original, non bipartite, Harer-Zagier formula.

\end{document}